\documentclass{article}
\usepackage[utf8]{inputenc}
\usepackage{tikz}
\usetikzlibrary{calc}

\usepackage{amsmath, amsthm, amssymb, tikz, hyperref, enumerate}
\usepackage{amsfonts}
\usepackage{fullpage}
\usepackage[enableskew,vcentermath]{youngtab}
\definecolor{amber}{rgb}{1.0,0.75,0.0}
\definecolor{applegreen}{rgb}{0.55,0.71,0.0}
\definecolor{byzantium}{rgb}{0.44, 0.16, 0.39}
\definecolor{cadmiumorange}{rgb}{0.93, 0.53, 0.18}
\definecolor{darkcyan}{rgb}{0.0, 0.55, 0.55}
\newcommand\beq{\begin{equation}}
\newcommand\eeq{\end{equation}}
\newcommand\bce{\begin{center}}
\newcommand\ece{\end{center}}
\newcommand\bea{\begin{eqnarray}}
\newcommand\eea{\end{eqnarray}}
\newcommand\ba{\begin{array}}
\newcommand\ea{\end{array}}
\newcommand\ben{\begin{enumerate}}
\newcommand\een{\end{enumerate}}
\newcommand\bit{\begin{itemize}}
\newcommand\eit{\end{itemize}}
\newcommand\brr{\begin{array}}
\newcommand\err{\end{array}}
\newcommand\bt{\begin{tabular}}
\newcommand\et{\end{tabular}}

\newtheorem{theorem}{Theorem}[section]

\newtheorem{definition}[theorem]{Definition}
\newtheorem{example}[theorem]{Example}

\newtheorem{defn}[theorem]{Definition}

\newtheorem{remark}[theorem]{Remark}

\newtheorem{observation}[theorem]{Observation}

\title{Separators - a new statistic for permutations}
\author{Eli Bagno, Estrella Eisenberg, Shulamit Reches and Moriah Sigron}
\date{\today}

\begin{document}

\maketitle
\abstract
A digit $\pi_j$ in a permutation $\pi=[\pi_1,\ldots,\pi_n]\in S_n$ is defined to be a separator of $\pi$ if by omitting it from $\pi$ we get a new $2-$block. In this work we introduce a new statistic, the number of separators, on the symmetric group $S_n$ and calculate its distribution over $S_n$. We also provide some enumerative and asymptotic results regarding this statistic. 

\section{Introduction}
Let $S_n$ be the Symmetric group of $n$ elements. Let $\sigma,\pi \in \bigcup_{n \in \mathbb{N}}{S_n}$. We say that $\sigma$ {\it contains} $\pi$ if there is a sub-sequence of elements of $\sigma$ that is order-isomorphic to $\pi$.
As an example, the permutation $\sigma=[3624715]$ (written in one-line-notation) contains $\pi=[3142]$ as both the sub-sequences $6275$ and $6475$ testify. 
If $\pi$ is contained in $\sigma$, then we write $\pi \preceq \sigma$.
The set of all permutations $\cup_{n \in \mathbb{N}} S_n$ is a poset under the partial order given by containment. This is called the permutation pattern poset.

\begin{example}
Let $\sigma=[53241] \in S_5$. In order to find the permutations in $S_4$ which are contained in $\sigma$, we shall remove each of the digits of $\pi$ and standardize. If we remove $\sigma_1=5$, we get the permutation $[3241]$, while if we remove $\sigma_2=3$ or $\sigma_3=2$ we get the permutation $[4231]$. The removal of $\sigma_4=4$ produces the permutation $[4321]$, and the removal of $\sigma_4=1$ produces the permutation $[4213]$. The situation can also be read from the the following picture:

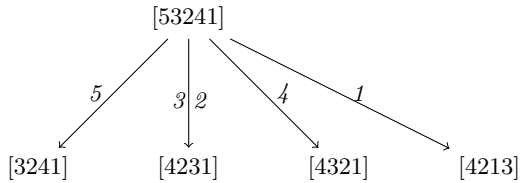
\begin{tikzpicture}
\label{}
        \tikzstyle{every node} = [rectangle]
        
        \node (53241) at (0,0) {$[53241]$};
        \node (3241) at (-2,-2) {$[3241]$};
        \node at (-1.25,-1) {5};
        \node (4231) at (0,-2) {$[4231]$};
        \node at (0,-1.1) {3 2};
        \node (4321) at (2,-2) {$[4321]$};
        \node at (1.25,-1) {4};
        \node (4213) at (4,-2) {$[4213]$};
        \node at (2.25,-1) {1};

        \foreach \from/\to in {53241/3241,53241/4231,53241/4321,53241/4213}
            \draw[->] (\from) -- (\to);
    \end{tikzpicture}

\end{example}


In the example above, the removal of the digits $2$ and $3$ produced the same permutation. This is not coincidental. It happens since these two digits form a 2-block in $\pi$ as we define below:  (See also in \cite{H2}, Definition 4)
\begin{defn}

Let $\pi =[\pi_1,\ldots,\pi_n]$ be a permutation and let $i \in [n-1]$. We say that the pair $(\pi_i,\pi_{i+1})$ is a {\em 2-block} or a {\em bond} in $\pi$ if $\pi_{i}-\pi_{i+1} =\pm 1$ . We say that
the sequence $(\pi_{i}, \pi_{i+1},\ldots ,\pi_{i+k-1})$ is a {\em run} of length $k>1$ if, for $0 \leq j \leq k-2$, the pair $(\pi_{i+j},\pi_{i+j+1})$ is a {\rm bond}. We allow also (trivial) runs of length $k=1$.
Note that a run of a permutation might be ascending or descending.
Occasionally, we omit the parentheses when we write blocks or runs.

\end{defn}
\begin{example}
The permutation $\pi=[45187623]$ has $45$, $1$, $876$ and $23$ as its maximal runs. 
\end{example}
The distribution of the bonds has been examined previously in \cite{Wolfowitz, Kaplansky, H2}.  Each run of length $n \geq 1$ contains $n-1$ bonds.
 The number of bonds in a permutation $\sigma \in S_n$ affects the structure of the poset of all permutations contained in $\sigma$, the downset of $\sigma$. This happens since the number of permutations $\pi \in S_{n-1}$ such that $\pi \preceq \sigma$  is $n-\beta(\sigma)$, where $\beta(\sigma)$ is the number of bonds in $\sigma$. (See Theorem 6 in \cite{H2}).

To better understand the structure of the poset $\bigcup_{n \in \mathbb{N}}{S_n}$, we would like to get information not only about the number of bonds of a given $\sigma \in S_n$, but also about the number of bonds of the permutations contained in $\sigma$. Hence we introduce a new concept: A digit of a permutation, a removal of which produces a {\bf new} bond, will be called a {\it separator}. (see the formal definition below).  

\begin{example}
In the permutation $\pi=[567139482]$ we can omit $\pi_7=4$ and after standardizing we get the permutation $[45613872]$ which has the {\bf new} $2$-block $87$, so $4$ is a separator. The digit $2$ is also a separator of $\pi$, since the removal of it creates the permutation $[45612837]$ containing the $\textbf{new}$ bond $12$. Note that if we remove $\pi_2=6$ from $\pi$, we get the permutation $[56138472]$ which contains the bond $56$ that already exists in $\pi$, so $6$ is not a separator in $\pi$.
\end{example}

Formally:

\begin{definition}\label{def separate}
For $\sigma=[\sigma_1,\ldots,\sigma_n] \in S_n$ we say that ${\bf \textcolor{red}{\sigma_i}}$ {\em separates} $\sigma_{j_1}$ from $\sigma_{j_2}$ in $\sigma$ if by omitting $\sigma_i$ from $\sigma$ we get a {\bf new} $2-$block. This happens if and only if one of the following cases holds: 

\begin{enumerate}
\item $j_1,i,j_2$ are subsequent numbers and $|\sigma_{j_1}-\sigma_{j_2}|=1$, i.e $\sigma_i=b$  and $$\sigma=[\ldots,{\bf a},{\bf \textcolor{red}{b}},{\bf a \pm 1},\ldots]$$ We call $\sigma_i$  a separator of type $I$ or a vertical separator.
\item  $\sigma_{j_1},\sigma_i,\sigma_{j_2}$ are subsequent numbers and $|j_1-j_2|=1$, i.e $\sigma_i=a$ and $$\sigma=[\ldots,{\bf \textcolor{red}{a}},\ldots,{\bf a \pm1,a\mp1},\ldots]$$ or $$\sigma=[\ldots,{\bf a\pm1,a\mp1},\ldots,{\bf \textcolor{red}{a}},\ldots].$$
We call $\sigma_i$ a separator of type $II$ or a horizontal separator.
\end{enumerate}

\end{definition}
The choice of the names of the separators of types $I$ and $II$ is explained in the following picture in which $\sigma_1=3$ is a horizontal separator (which its omitting forms the 2-$block$  $23$), $\sigma_3=5$ is a vertical separator (which its omitting forms the 2-$block$  $12$) and $\sigma_4=2$ is both (which its omitting forms the two 2-$blocks$  $21$ and $43$).   

\textcolor{amber}{Horizontal separator} and \textcolor{blue}{Vertical separator} and \textcolor{applegreen}{separators of both types} .

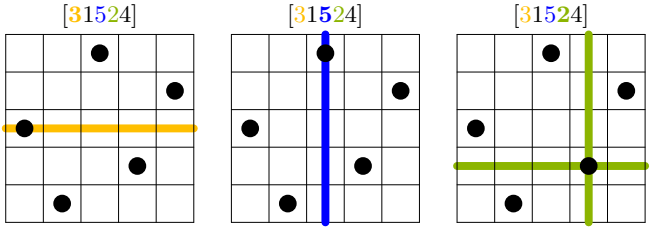
\begin{tikzpicture}[cap=round,line width=0.2pt,scale=0.5]

  \node (31524) at (2.5,5.5) {$[\textcolor{amber}{\bf3}1\textcolor{blue}{5}\textcolor{applegreen}{2}4]$};

\draw[-](0,0)--(5,0);
\draw[-](0,1)--(5,1);
\draw[-](0,2)--(5,2);
\draw[-](0,3)--(5,3);
\draw[-](0,4)--(5,4);
\draw[-](0,5)--(5,5);
\draw[amber,line width=3pt](0,2.5)--(5,2.5);
\draw[-](0,0)--(0,5);
\draw[-](1,0)--(1,5);
\draw[-](2,0)--(2,5);
\draw[-](3,0)--(3,5);
\draw[-](4,0)--(4,5);
\draw[-](5,0)--(5,5);

\draw[line width=1pt][fill]  (0.5,2.5) circle (0.2cm);
\draw[line width=1pt][fill]  (1.5,0.5) circle (0.2cm);
\draw[line width=1pt][fill]  (2.5,4.5) circle (0.2cm);
\draw[line width=1pt][fill]  (3.5,1.5) circle (0.2cm);
\draw[line width=1pt][fill]  (4.5,3.5) circle (0.2cm);

\node (31524) at (8.5,5.5) {$[\textcolor{amber}{3}1\textcolor{blue}{\bf5}\textcolor{applegreen}{2}4]$};

\draw[-](6,0)--(11,0);
\draw[-](6,1)--(11,1);
\draw[-](6,2)--(11,2);
\draw[-](6,3)--(11,3);
\draw[-](6,4)--(11,4);
\draw[-](6,5)--(11,5);
\draw[blue,line width=3pt](8.5,0)--(8.5,5);
\draw[-](6,0)--(6,5);
\draw[-](7,0)--(7,5);
\draw[-](8,0)--(8,5);
\draw[-](9,0)--(9,5);
\draw[-](10,0)--(10,5);
\draw[-](11,0)--(11,5);

\draw[line width=1pt][fill]  (6.5,2.5) circle (0.2cm);
\draw[line width=1pt][fill]  (7.5,0.5) circle (0.2cm);
\draw[line width=1pt][fill]  (8.5,4.5) circle (0.2cm);
\draw[line width=1pt][fill]  (9.5,1.5) circle (0.2cm);
\draw[line width=1pt][fill]  (10.5,3.5) circle (0.2cm);

\node (31524) at (14.5,5.5) {$[\textcolor{amber}{3}1\textcolor{blue}{5}\textcolor{applegreen}{\bf2}4]$};

\draw[-](12,0)--(17,0);
\draw[-](12,1)--(17,1);
\draw[-](12,2)--(17,2);
\draw[-](12,3)--(17,3);
\draw[-](12,4)--(17,4);
\draw[-](12,5)--(17,5);
\draw[applegreen,line width=3pt](15.5,0)--(15.5,5);
\draw[applegreen,line width=3pt](12,1.5)--(17,1.5);
\draw[-](12,0)--(12,5);
\draw[-](13,0)--(13,5);
\draw[-](14,0)--(14,5);
\draw[-](15,0)--(15,5);
\draw[-](16,0)--(16,5);
\draw[-](17,0)--(17,5);

\draw[line width=1pt][fill]  (12.5,2.5) circle (0.2cm);
\draw[line width=1pt][fill]  (13.5,0.5) circle (0.2cm);
\draw[line width=1pt][fill]  (14.5,4.5) circle (0.2cm);
\draw[line width=1pt][fill]  (15.5,1.5) circle (0.2cm);
\draw[line width=1pt][fill]  (16.5,3.5) circle (0.2cm);


\end{tikzpicture}

\begin{defn}
Let $Sep_V(\pi)$ and $Sep_H(\pi)$ be the sets of vertical and horizontal separators of a permutation $\pi$ respectively. Let $Sep(\pi)=Sep_V(\pi) \cup Sep_{H}(\pi)$ and $sep(\pi)=|Sep(\pi)|$.
\end{defn}

\begin{example}
\label{ex_sep_types1and2}
Let $\sigma=[132465879]$.   Then $Sep_{V}(\sigma)=\{3,2,6,7\}$, and $Sep_{H}(\sigma)=\{3,2,5,8\}$. 
Note that $7$ is a vertical separator, even though $7$ is a part of a $2-$block: $87$, since by omitting $7$ from $\sigma$ we get a $\bf{new}$ $2-$block: $78$. 
\end{example}

\begin{remark}\label{obs on Sep}
Several comments are now in order:
\begin{enumerate}

\item Notice the significance of the word 'new' in Definition \ref{def separate}. For example, the identity permutation has plenty of $2$-blocks even though it has no separators.
\item The numbers $1$ and $n$ can only be vertical separators, while $\sigma_1$ and $\sigma_n$ can only be horizontal separators.
\item If $\sigma_i$ is a vertical separator in $\sigma$ then $i$ is a horizontal separator in $\sigma^{-1}$. Hence $Sep_V(\sigma)=Sep_{H}(\sigma^{-1})$

\item $Sep_V(\sigma)=Sep_V(\sigma^r)$ and  $Sep_H(\sigma)=Sep_{H}(\sigma^r)$ where $\sigma^r$ is the reverse of $\sigma$.  
\item A separator can be of both types, vertical and horizontal. For instance,  in example \ref{ex_sep_types1and2}, the digits $2,3$ are separators of both types.
\end{enumerate} 
\end{remark}

Permutations of $S_n$ which have no bonds are connected to the problem of placing $n$ non-attacking kings in an $n \times n$ chess board. These  permutations were counted in 
\cite{Ro}, and the structure of their containment poset is discussed in a recent paper by the authors 
of this note \cite{BERS}. The set of such permutations will be denoted by $K_n$. 

If $\sigma \in K_n$ then even though $\sigma$ has no bonds, after omitting a digit from $\sigma$, the resulting permutation might have (at least) one. Recall that in this case the omitted digit is a separator of $\sigma$.
The connection between the number of separators in $\sigma$ and the number of $\pi \in K_{n-1}$ such that $\pi \preceq \sigma$ is given by the following:

\begin{observation}

Let $\sigma \in K_n$. Then the number of $\pi \in K_{n-1}$ such that $\pi \preceq \sigma$  is $n-sep(\sigma)$ where $sep(\sigma)$ is the number of separators in  $\sigma$. 
\end{observation}

In \cite{H2}, Homberger built a multivariate generating function which presented the distribution of the bonds throughout all permutations. He used the principal of inclusion-exclusion in its generating function version. 
The main result of this paper uses the same method for producing a multivariate generating function representing the distribution of the vertical separators (and thus also of horizontal separators, by Remark \ref{obs on Sep}.3). 

\section{Permutations with no separators and permutations with maximal number of separators}
The permutations in $S_n$ that have no separators of any type, are counted by the sequence
A137774 from OEIS. They correspond to the number of ways to place $n$ non-attacking empresses (a chess piece which moves like a rook and a night) on an $n \times n$ chess board. 

Theorem \ref{number of perm with n sep} below deals with the opposite case, i.e., the number of permutations, all the digits of which are separators.
First, we have to include some definitions. A comprehensive survey of these concepts can be found in \cite{Vatter}.

\begin{definition}
Let $\pi=[\pi_1, \ldots ,\pi_n] \in S_n$.   
A {\em block} (or {\em interval}) of $\pi$ is a nonempty contiguous
sequence of entries $\pi_i \pi_{i+1} \ldots \pi_{i+k}$ whose values also form a contiguous sequence of integers.
\end{definition}

\begin{example}
If $\pi = [2647513]$ then $6475$ is a block but $64751$ is not. 
\end{example}

Each permutation can be decomposed into singleton blocks, and also forms a single block by itself; these are the {\em trivial blocks} of the permutation. All other blocks are called {\em proper}.


\begin{definition}
A {\em block decomposition} of a permutation is a partition of it into disjoint blocks. 
\end{definition}

For example, the permutation $\sigma=[67183524]$  can be decomposed as $67\ 1\ 8\ 3524$. 
In this example, the relative order between the blocks forms the permutation $[3142]$, i.e., if we take for each block one of its digits as a representative then the sequence of representatives is order-isomorphic to $[3142]$. 
Moreover, the block $67$ is order-isomorphic to $[12]$, and the block $3524$ is order-isomorphic to $[2413]$. These are instances of the concept of {\em inflation}, defined as follows.

\begin{definition}
\label{inflation}
Let $n_1, \ldots, n_k$ be positive integers with $n_1 + \cdots + n_k = n$.
The {\em inflation} of a permutation $\pi \in S_k$ by the permutations $\alpha_i \in S_{n_i}$ $(1 \leq i \leq k)$ is 
the permutation $\pi[\alpha_1, \ldots, \alpha_k] \in S_n$ obtained by replacing the $i$-th entry of $\pi$ by a block which is order-isomorphic to the permutation $\alpha_i$
on the numbers $\{s_i + 1, \ldots, s_i + n_i\}$ instead of $\{1, \ldots, n_i\}$, where $s_i = n_1 + \cdots + n_{i-1}$ $(1 \leq i \leq k)$. 
\end{definition}

\begin{example}
The inflation of $[2413]$ by $[213],[21],[132]$ and $[1]$ is 
\[
2413[213,21,132,1]=[546 \ 98 \ 132 \ 7].
\]
\end{example}

We are interested in the structure of all permutations in $S_n$ in which every digit is a separator. 
In Theorem 3.18 in \cite{BERS} we proved that in a permutation $\sigma \in K_n$ each digit of $\sigma$ is a separator if and only if $\sigma=\pi[\alpha_1,\dots,\alpha_k]$ where $\alpha_1,\dots, \alpha_k \in \{[3142],[2413]\}$ and $\pi \in S_k$. 

In the following Theorem, we extend this result and show that this structure holds for each permutation in $S_n$ in which each one of its digits is a separator. 

\begin{theorem}
\label{sep}
 In a permutation $\sigma \in S_n$, each digit is a separator if and only if  $n=4k,\, k \in \mathbb{N}$ and there are $\alpha_1,\dots, \alpha_k \in \{[3142],[2413]\}$ and $\pi \in S_k$ such that $\sigma=\pi[\alpha_1,\dots,\alpha_k]$.
\end{theorem}
\begin{proof}
The "only if" side is obvious, so we will prove only the "if" side. Let $\sigma \in S_n$ be a permutation such that each
digit of $\sigma$ is a separator. 
If we show that $\sigma$ has no $2$-block, i.e., $\sigma \in K_n$, then by Theorem 3.18 in \cite{BERS}, we are done. 
We assume to the contrary that $\sigma$ contains a block of the form $a,a+1$ and show that not all the digits of $\sigma$ are separators. 

We divide in two different cases, according to the type of separation of $a+1$

\begin{itemize}
\item $a+1$ is a vertical separator:
In this case, $\sigma$ contains the sub-sequence $\cdots a,a+1,a-1\cdots $. The digit $a-1$ is also a separator, so we distinguish between two cases according to the type of separation of $a-1$. 

\begin{enumerate}
    \item If $a-1$ is a vertical separator then $\sigma=\cdots a,a+1,a-1,a+2 \cdots $. Hence $a+2$ must be a vertical separator so we have $\sigma=\cdots a,a+1,a-1,a+2,a-2\cdots$. By the same argument, $a-2$ must be a vertical separator and so $\sigma=\cdots a,a+1,a-1,a+2,a-2,a+3\cdots $. This process continues until we reach $\sigma_n$ which can not be a horizontal separator but also can not be a vertical separator. 
    
    \item If $a-1$ is a horizontal separator then $\sigma=\cdots a-2,a,a+1,a-1 \cdots $. Hence $a-2$ must be a horizontal separator so that  $\sigma=\cdots a-2,a,a+1,a-1,a-3\cdots$. By the same argument, $a-3$ must be a horizontal separator and so $\sigma=\cdots a-4, a-2,a,a+1,a-1,a-3\cdots $. This process continues until we reach $a-k=1$ which can not be of a vertical separator but also can not be of a horizontal separator. 
    
\end{enumerate}

\item $a+1$ ia a horizontal separator: In this case, we consider $\sigma^{-1}$ in which $(\sigma^{-1})_{a+1}$ is a vertical separator.
Since $a,a+1$ is a block in $\sigma$, $\sigma^{-1}$  contains a block in locations $a,a+1$, which means that $(\sigma^{-1})_a=b,(\sigma^{-1})_{a+1}=b + 1$ for some $b \in \{1, \ldots, n-1 \}$. 
Now, by remark \ref{obs on Sep}.3, we can apply the argument of the previous case in order to show that not all of the digits of $\sigma^{-1}$ are separators. This implies that not all the digits of $\sigma$ are separators and we are done. 
\end{itemize}

We conclude that $\sigma \in K_n$ where $K_n$ is the set of king permutations of order $n$. 
\end{proof}

In Figure \ref{fig1} we can see the structure of such  permutations that each one of their digits is a separator, according to the above theorem.

\begin{figure}[!ht]

  \centering
  \begin{tikzpicture}[cap=round,line width=0.2pt,scale=0.3]
\draw[-](0,0)--(16,0);
\draw[-](0,1)--(16,1);
\draw[-](0,2)--(16,2);
\draw[-](0,3)--(16,3);

\draw[-](0,4)--(16,4);
\draw[-](0,5)--(16,5);
\draw[-](0,6)--(16,6);
\draw[-](0,7)--(16,7);

\draw[-](0,8)--(16,8);
\draw[-](0,9)--(16,9);
\draw[-](0,10)--(16,10);
\draw[-](0,11)--(16,11);

\draw[-](0,12)--(16,12);
\draw[-](0,13)--(16,13);
\draw[-](0,14)--(16,14);
\draw[-](0,15)--(16,15);
\draw[-](0,16)--(16,16);

\draw[-](0,0)--(0,16);
\draw[-](1,0)--(1,16);
\draw[-](2,0)--(2,16);
\draw[-](3,0)--(3,16);

\draw[-](4,0)--(4,16);
\draw[-](5,0)--(5,16);
\draw[-](6,0)--(6,16);
\draw[-](7,0)--(7,16);

\draw[-](8,0)--(8,16);
\draw[-](9,0)--(9,16);
\draw[-](10,0)--(10,16);
\draw[-](11,0)--(11,16);

\draw[-](12,0)--(12,16);
\draw[-](13,0)--(13,16);
\draw[-](14,0)--(14,16);
\draw[-](15,0)--(15,16);
\draw[-](16,0)--(16,16);

\draw[applegreen,line width=4pt](0,0+16)--(4,0+16);
\draw[applegreen,line width=4pt](0,0+16)--(0,0+12);
\draw[applegreen,line width=4pt](0+4,0+12)--(0+4,0+16);
\draw[applegreen,line width=4pt](0+4,0+12)--(0,0+12);

\draw[red,line width=4pt](4,8)--(8,8);
\draw[red,line width=4pt](4,8)--(4,4);
\draw[red,line width=4pt](8,4)--(8,8);
\draw[red,line width=4pt](8,4)--(4,4);

\draw[blue,line width=4pt](8,4)--(12,4);
\draw[blue,line width=4pt](8,4)--(8,0);
\draw[blue,line width=4pt](12,0)--(12,4);
\draw[blue,line width=4pt](12,0)--(8,0);

\draw[amber,line width=4pt](12,12)--(16,12);
\draw[amber,line width=4pt](12,12)--(12,8);
\draw[amber,line width=4pt](16,8)--(16,12);
\draw[amber,line width=4pt](16,8)--(12,8);

\draw[line width=1pt][fill]  (0.5,13.5) circle (0.3cm);
\draw[line width=1pt][fill]  (1.5,15.5) circle (0.3cm);
\draw[line width=1pt][fill]  (2.5,12.5) circle (0.3cm);
\draw[line width=1pt][fill]  (3.5,14.5) circle (0.3cm);

\draw[line width=1pt][fill]  (4.5,6.5) circle (0.3cm);
\draw[line width=1pt][fill]  (5.5,4.5) circle (0.3cm);
\draw[line width=1pt][fill]  (6.5,7.5) circle (0.3cm);
\draw[line width=1pt][fill]  (7.5,5.5) circle (0.3cm);

\draw[line width=1pt][fill]  (8.5,1.5) circle (0.3cm);
\draw[line width=1pt][fill]  (9.5,3.5) circle (0.3cm);
\draw[line width=1pt][fill]  (10.5,0.5) circle (0.3cm);
\draw[line width=1pt][fill]  (11.5,2.5) circle (0.3cm);

\draw[line width=1pt][fill]  (12.5,10.5) circle (0.3cm);
\draw[line width=1pt][fill]  (13.5,8.5) circle (0.3cm);
\draw[line width=1pt][fill]  (14.5,11.5) circle (0.3cm);
\draw[line width=1pt][fill]  (15.5,9.5) circle (0.3cm);

\end{tikzpicture}

\caption{the plot of $[\textcolor{applegreen}{14,16,13,15},\textcolor{red}{7,5,8,6},\textcolor{blue}{2,4,1,3},\textcolor{amber}{11,9,12,10}]$}\label{fig1}

\end{figure}

According to Theorem \ref{sep}, we can now enumerate those permutations.
\begin{theorem}
\label{number of perm with n sep}
The number of permutations in $S_n$ which have exactly $n$ different separators is:

$$
\begin{cases}
2^k k! & n=4k  \\
 0  &{\text O.W}.
\end{cases} \\$$

\end{theorem}

\section{A generating function for vertical separators}

In this section we present a generating function for the number of vertical separators. 
 For each $n,m \in \mathbb{N}$ let $s_{n,m}$ be the number of permutations $\pi \in S_n$ with exactly $m$ vertical separators. We want to calculate the generating function 
 $h(z,u)=\sum\limits_{n \geq 0}{\sum\limits_{m=0}^{n}s_{n,m}z^nu^m}$. According to remark \ref{obs on Sep}.3, this  generating function is the same for horizontal separators.\\

 In order to construct the function $h(z,u)$, we will enlarge the set of elements we work with such that it will contain marked permutations. We then use the principle of inclusion-exclusion together with a method of splitting permutations into two parts to achieve the generating function for the number of vertical separators

\subsection{Counting permutation with mark bonds} 

 A {\it marked permutation} is a permutation in which each bond can be chosen to be marked or not. 
 The marked bonds will be denoted by a bar above the corresponding part of the permutation. If several adjacent bonds are marked, then we put a long bar above the corresponding run. An entry that is not contained in a marked bond is considered to be a run of length $1$. 
\begin{example}
Let $\pi=[6 1 3 4 5 2 8 7 9]$. Here are some permutations with marked bonds, made out of $\pi$: $[6 1 \overline{3 4 5} 2 \overline{8 7} 9]$, $[6 1 \overline{3 4} 5 2 8 7 9]$,  $[6 1 3 4 5 2 \overline{8 7} 9]$.  
\end{example}

In order to count the marked permutations, we introduce another way to write them. 

 Recall that for $n \in \mathbb{N}$, a {\it composition} with $m$ non-zero parts of $n$ is a vector $(a_1,\dots,a_m)$ such that $a_i \in \mathbb{N}$ and $\sum\limits_{i=1}^m{a_i}=n$. We define an {\it arrowed composition} of $n$ to be a composition in which after every part which is greater than $1$ there exists one of the signs $\uparrow$ or $\downarrow$. 

 For example, $(2\uparrow,1,7\downarrow,2\uparrow)$ is an arrowed composition of $n=12$. 

Now, each marked permutation $\pi \in S_n$ can be uniquely presented as an arrowed composition $(a_1,\dots,a_m)$ of $n$, together with a permutation $\sigma \in S
_m$. This idea will be best clarified by the following example. 

\begin{example}
Let $\pi=[2\overline{45}6 1\overline{987}3]$. We write $\pi$ as a pair consisting of an arrowed composition of $m=6$ parts $\lambda$, and a permutation $\sigma \in S_6$.   First, write $\pi$ as a sequence of runs: $b_1=2,b_2=\overline{45},b_3=6,b_4=1,b_5=\overline{987},b_6=3$.
Each run contributes its length to the composition. Then for each part, we add the sign $\uparrow$ if the corresponding run is increasing, the sign $\downarrow$ if the run is decreasing and no arrow if the run is of length $1$. 
In our case we get $\lambda=(1,2\uparrow,1,1,3\downarrow,1)$. Now $\sigma \in S_6$ is the permutation induced by the order of the blocks. In our case we have: $\sigma=[245163]$. The marked permutation $\pi$ is now uniquely defined by the pair $(\lambda,\sigma)$. 
\end{example}

In other words, if we replace each $j\uparrow$ with the ascending permutation $[123\ldots j]$ and each $j\downarrow$ with the descending permutation $[j\ldots 321]$, we can see that this defines an inflation. In the previous example we can write $\pi=245163[1,12,1,1,321,1]$.
 For convenience, we denote this inflation by  $\sigma[\lambda]$.\\
 
 In \cite{H2} (during the proof of Theorem 10), the author calculated the generating function, counting the number of permutations having a specific number of bonds. This was done by  calculating the generating function of marked bonds, and using the inclusion-exclusion principle. If we denote by $a_{n,m}$ the number of permutations of $S_n$ with $m$ marked bonds and put $A(z,u)=\sum\limits_{n\geq 1}\sum\limits_{m\geq 0}{a_{n,m}z^n u^m}$, then the identity permutation contributes $z$ and for each $j\geq 2$, a run of order $j$ can be either $[123\ldots j]$ or $[j\ldots 321]$, each of them contributes $u^{n-1}$, so the contribution is $2z^j u^{j-1}$. It is easy to see from the above that 
$$A(z,u)=\sum_{m \geq 0}{m!(z+2z^2u+2z^3u^2+2z^4u^3+\cdots)^m=\sum_{m\geq 0}{m! (z+\frac{2z^2u}{1-zu})^m}}.$$ 
(Here $m$ denotes the number of runs). \\

\subsection{Comb  decomposition and marked separators}
Coming back to our counting of permutations with respect to the number of vertical separators, we show now 
how to make a reduction  of this problem to the problem of counting marked permutations with respect to the number of bonds. 
We start with the definition of what we call here  {\it comb permutations} as follows:
 \begin{definition}\label{Hadamar}
 Let $\sigma=(\sigma_1,\dots,\sigma_k),\tau=(\tau_1,\dots,\tau_k)$ be two sequences such that $\{\sigma_1,\dots,\sigma_k,\tau_1,\dots,\tau_k\}=\{1,2,\dots, 2k\}$. Define the comb permutation $\pi=\sigma \odot \tau$ by $$\pi=[\sigma_1,\tau_1,\sigma_2,\tau_2,\ldots ,\sigma_k,\tau_k] \in S_{2k}.$$ 
 
 Similarly, let $\sigma=(\sigma_1,\dots,\sigma_{k+1}),\tau=(\tau_1,\dots,\tau_k)$ be two sequences such that $\{\sigma_1,\dots,\sigma_{k+1},\tau_1,\dots,\tau_k\}=\{1,2,\dots, 2k+1\}$. We define the comb permutation $\pi=\sigma \odot \tau$ by $$\pi=[\sigma_1,\tau_1,\sigma_2,\tau_2,\dots ,\sigma_k,\tau_k, \sigma_{k+1}] \in S_{2k+1}.$$ 
 When $\pi=\sigma \odot \tau$, we denote $\sigma$ as $\pi^{odd}$ and $\tau$ as $\pi^{even}$.
 \end{definition}

 Now, let $\pi=\pi^{odd} \odot \pi^{even}$ where $\pi^{odd}$ and $\pi^{even}$ are sequences with marked bonds.  Note that if $(\pi^{odd}_{i}\pi^{odd}_{i+1})$ is a marked bond of $\pi^{odd}$ then the element of $\pi^{even}$ which lies between $\pi^{odd}_{i}$ and $\pi^{odd}_{i+1}$ in $\pi$ is a vertical separator, we call it a {\it marked separator} and denote it by putting a hat over it. For example, if $\pi^{odd}=(\overline{12}4)$ and $\pi^{even}=(53)$ then $\overline{12}$ is a marked bond in $\pi^{odd}$ and thus $5$ is a marked separator in $\pi^{odd}\odot \pi^{even}=1\hat{5}234$. Similarly, define marked separators for marked bonds in $\pi^{even}$. Therefore, we have the following: 
\begin{observation}\label{split}

 Let $\pi^{odd},\pi^{even}$ be defined as in Definition \ref{Hadamar} and let $\pi=\pi^{odd} \odot \pi^{even}$. Then the number of (marked) vertical separators in $\pi$ is equal to the total number of (marked) bonds in $\pi^{odd}$ and in $\pi^{even}$. 
 \end{observation}
 
 
Let $n=2k$. Given two arrowed compositions of $k$: $\lambda_o$ of size $m_o$ and $\lambda_e$ of size  $m_e$, 
and given a permutation $\sigma \in S_{m_o+m_e}$, we take the inflation $\alpha=\sigma[\lambda_o,\lambda_e]$ (it is a permutation with marked bonds).  We construct a permutation $\pi$ as follows: denote the first $k$ elements of $\alpha$ by $\pi^{odd}$ and the last $k$ elements of $\alpha$ by $\pi^{even}$. Now $\pi=\pi^{odd}\odot\pi^{even}$. The case $n=2k+1$ is similar.

\begin{example}
\label{exa}
For $k=4$, let $\lambda_o=(1,3\downarrow)$,  $\lambda_e=(1,1,2\uparrow)$ and let $\sigma=[34215]$. Then \newline $\alpha=34215[1,321,1,1,12]=[3\overline{654}21\overline{78}]$. Thus $\pi^{odd}=(3\overline{654})$ and $\pi^{even}=(21\overline{78})$ and therefore $\pi=(3\overline{654})\odot (21\overline{78})=[326\hat{1}5\hat{7}\hat{4}8]$.
\end{example}

On the other side, given a permutation $\pi \in S_{n}$ with marked separators, let $\pi^{odd}$ be the sequence of the odd entries of $\pi$ and $\pi^{even}$ be the sequence of the even entries of $\pi$, and mark the relevant bonds, i.e. if $\pi_i$ is a marked separator in $\pi$, then $(\pi_{i-1}\pi_{i+1})$ is a marked bond in $\pi^{odd}$ or in $\pi^{even}$.
Denote by $\alpha$ the permutation with marked bonds obtained by $\pi^{odd}$ followed by $\pi^{even}$. We know that $\alpha$ can  be uniquely presented as an arrowed composition of $n$ together with a permutation of the number of parts, $m$. 
 
\begin{example}
$\pi=[27\hat{1}86\hat{3}549]$.
We use the sequences $\pi^{odd}=(21\overline{65}9)$ and $\pi^{even}=(\overline{78}34)$ to produce $\alpha=[21\overline{65}9\,\overline{78}34]$. This permutation can be presented as $\lambda=(1,1,2\downarrow,1,2\uparrow,1,1)$ with $\sigma=[2157634]$.
\end{example}


\subsection{Calculating the generating function for vertical separators }
Recall that our goal is to find the function $h(z,u)$, which is the generating function for the number of vertical separators.
In order to do that, we first calculate the generating function for the number of {\bf marked} vertical separators. Denote by $b_{n,m}$ the number of permutations of $S_n$ with $m$ marked vertical separators, and let $g(z,u)=\sum\limits_{n\geq 1}\sum\limits_{m\geq 0}{b_{n,m}z^n u^m}$.

As we saw below, there is a correspondence between the number of marked bonds and the number of marked vertical separators, thus we construct the generating function $g(z,u)$ by calculating separately the generating functions for the marked bonds of the odd and the even parts of each permutation. The requirement that the odd part and the even part of a permutation must have (almost) the same size will be met by using the well known {\it Hadamard product} (element-wise) of polynomials and series. 
 \begin{definition}
 Let $R$ be a ring and let $f(x)=\sum\limits_{n \in \mathbb{N}} a_nx^n$, $g(x)=\sum\limits_{n \in \mathbb{N}} {b_nx^n} \in R[[x]]$ be two power series in $x$.  The {\it Hadamard product} of $f(x),g(x)$ is $f(x)*g(x)=\sum\limits_{n=0}^\infty a_n b_n x^n$.
\end{definition}

\begin{example}
$(2+3x-4x^2) * (5+x+7x^2) = 10+3x-28x^2$.
\end{example}

 In order to form the generating function of the marked separators, $g(z,u)$, let us have a look at the permutations $\pi \in S_n$ for a fixed $n$. We would like to find the monomial contributed  by each $\pi=\pi^{odd}\odot\pi^{even} \in S_n$ using the monomial corresponding to the marked bonds of $\pi^{odd}$ and $\pi^{even}$. 
 
 Let $n=2k$, in this case, $\pi^{odd}, \pi^{even}$ are sequences of order $k$, and therefore we use the Hadamard product to combine the two monomials.   Each $\pi \in S_{2k}$ contributes to $g(z,u)$ a monomial of the form $u^m z^{2k}$ where the monomials corresponding to the marked bonds of $\pi^{odd}$ and $\pi^{even}$ are $f_o(z,u)=u^{m_1} z^{k}$ and  $f_e(z,u)=u^{m_2} z^{k}$, respectively, where $m=m_1+m_2$. 
 If we look at those monomials as functions of the variable $z$, then we can easily see that the coefficient of $z^{2k}$ in $g(z,u)$ should be the product of the coefficients of $z^k$ of the odd and even parts. So we have to set $z^2$ instead of $z$ in the monomoials of $\pi^{odd}$ and $\pi^{even}$ and get:  $$f_o(z^2,u)*f_e(z^2,u)=u^{m_1} (z^2)^{k}*u^{m_2} (z^2)^{k}=u^m z^{2k}.$$ 
 
 \begin{example}
Return to example \ref{exa}. The monomial corresponding to the marked bonds of $\pi^{odd}=(3\overline{654})$ is $f_o(z,u)=z^4u^2$. In the same way,
 the monomial corresponding to the marked bonds of $\pi^{even}=(21\overline{78})$ is $f_e(z,u)=z^4u$. Thus $f_o(z,u)*f_e(z,u)=z^4u^3$. However, the monomial corresponding to the marked separators of $\pi=(3\overline{654})\odot (21\overline{78})=[326\hat{1}5\hat{7}\hat{4}8]$ is supposed to be $z^8u^3$. Note that if we take  $f_o(z^2,u)*f_e(z^2,u)$ we get exactly what we need. 
\end{example}

 Similarly, if $n=2k+1$, then each $\pi=\pi^{odd}\odot\pi^{even} \in S_n$ contributes a monomial of the form $u^m z^{2k+1}$ where the monomial of $\pi^{odd}$ is $f_o(z,u)=u^{m_1} z^{k+1}$ and the monomial of $\pi^{even}$ is $f_e(z,u)=u^{m_2} z^{k}$, $m=m_1+m_2$. Now, $$\big[ \frac{1}{z}f_o(z^2,u) \big] * \big[ z\,f_e(z^2,u) \big]=\big[ \frac{1}{z}u^{m_1} (z^2)^{k+1} \big] * \big[ z\,u^{m_2} (z^2)^{k} \big] =u^m z^{2k+1}.$$

\begin{example}
Let $\pi=[2\hat{5}\hat{3}4176]$. 
Then $\pi^{odd}=(\overline{23}16)$ and 
$\pi^{even}=(\overline{54}7)$. Then $f_o(z,u)=uz^4$ and $f_e(z,u)=uz^3$. So 
$$\big[\frac{1}{z}f_o(z^2,u) \big] * \big[z f_e(z^2,u) \big]=[\frac{1}{z}u(z^2)^4]*[z u(z^2)^3]=u^2z^7$$ as required. 

\end{example}

Using the above explanations, and the generating functions version to the inclusion-exclusion principle, we get the following calculation of the generating function of vertical separators, $h(z,u)$.

\begin{theorem}
$$h(z,u)=\sum_{m_o,m_e\geq 0}{(m_o+m_e)! \big[(z^2+\frac{2z^4(u-1)}{1-z^2(u-1)})^{m_o}\big]*\big[(z^2+\frac{2z^4(u-1)}{1-z^2(u-1)})^{m_e}}\big]\\$$

$$+\sum_{m_o,m_e\geq 0}{(m_o+m_e)! \big[(z^2+\frac{2z^4(u-1)}{1-z^2(u-1)})^{m_o}\frac{1}{z}\big]*\big[(z^2+\frac{2z^4(u-1)}{1-z^2(u-1)})^{m_e}}z\big]$$

where $*$ is the Hadamard product in $\mathbb{Q}[[u]][[z]]$. 
\end{theorem}

\begin{proof}

Let us denote for each $m \in \mathbb{N}$: $$p_m(z,v)=(z+2z^2v+2z^3v^2+\cdots)^m=(z+\frac{2z^2v}{1-zv})^{m}$$ Then $p_m(z,v)$ counts the number of ways to construct an arrowed composition $\lambda$ of size $m$. 
We relate to $n$ even and $n$ odd separately.
In order to construct a permutation of $S_{2k}$ with marked separators, we have to choose two arrowed compositions of $k$: $\lambda_o$ of size $m_o$, and $\lambda_e$ of size $m_e$, we also choose a permutation $\sigma \in S_{m_o+m_e}$. It is easy to see that this contributes to our function $(m_o+m_e)!p_{m_o}(z^2,v) *p_{m_e}(z^2,v)$. 
For $S_{2k+1}$ we have $(m_o+m_e)!\big[\frac{1}{z}p_{m_o}(z^2,v))\big] *\big[z\,p_{m_e}(z^2,v)\big]$.
We go over all the values for $m_o$ and $m_e$ for both $n$ even and $n$ odd and we obtain the generating function of the marked vertical separators: 

$$g(z,v)=\sum_{m_o,m_e\geq 0}{(m_o+m_e)! \big[(z^2+\frac{2z^4v}{1-z^2v})^{m_o}\big]*\big[(z^2+\frac{2z^4v}{1-z^2v})^{m_e}}\big]\\$$

$$+\sum_{m_o,m_e\geq 0}{(m_o+m_e)! \big[(z^2+\frac{2z^4v}{1-z^2v})^{m_o}\frac{1}{z}\big]*\big[(z^2+\frac{2z^4v}{1-z^2v})^{m_e}}z\big]$$



Now, we can use this generating function to obtain $h(z, u)$. The variable $v$
represents the \textbf{marked} vertical separators, while $u$ is responsible for vertical separators. Since every
vertical separator can either be marked or unmarked, it follows that by replacing $v + 1$ by $u$ we obtain that the generating function of the vertical separators is $$h(z,u)=g(z,u-1).$$

\end{proof}

\section{The expectation of the number of separators}
In this section we calculate the expectation of the number of  separators in a randomly 
chosen permutation. In order to do that, let us first calculate the expectation of the number of vertical separators. Consider the sample space of all $n!$ permutations, with the uniform probability. 
For each $1\leq i \leq n$, let $X_i$ be the Bernuli random variable such that for each $\pi\in S_n$, $X_i=1$ if 
the digit $i$ is a vertical separator in $\pi$ and $X_i=0$ otherwise. Then the sum $X=\sum\limits_{i=1}^n X_i$ counts the number of vertical separators for each $\pi \in S_n$.

In order to calculate $E[X]=\sum\limits_{i=1}^n{E[X_i]}$, let us first calculate $E[X_i]$ for each $1 \leq i \leq n$. 
The digit $i$ is a vertical separator in a permutation $\pi \in S_n$ if $\pi$ contains 
a consecutive sequence of the form $a,i,a+1$ or its reverse. 
If  $i\in \{1,n\}$, the digit $a$ can be chosen in $n-2$ ways. After choosing $a$ we have $(n-2)!$ ways to arrange the rest of the permutation, so that $E[X_1]=E[X_n]=\frac{2(n-2)(n-2)!}{n!}$. 
Now, for each $1<i<n$, the same consideration applies, but now we have only $(n-3)$ ways to choose $a$. This gives us $E[X_i]=\frac{2(n-3)(n-2)!}{n!}$. 
We have now: $$E[X]=\sum\limits_{i=1}^n{E[X_i]}=2E[X_1]+(n-2)E[X_2]$$

A simple calculation now yields the following:

\begin{theorem}
The expectation of the number of vertical separators in a randomly chosen $n$-permutation is $\frac{2(n-2)}{n}$. The asymptotic value of the expectation is $2$. 
\end{theorem}

We turn now to the calculation of the expectation of the number of separators that are both vertical and horizontal. Let $Y$ be the random variable counting the number of digits of a permutation $\pi$ which are separators of both types. Define, for each $1<i<n$, $Y_i$ to be the Bernuli 
random variable which for each $\pi\in S_n$ is equal to $1$ if the digit $i$ is both a vertical and a horizontal separator in $\pi$ and is equal to $0$ otherwise.  Note that the digits $1$ and $n$ can not be separators of both types. Let us calculate $E[Y_i]$ for $1<i<n$:

We have one of the following cases, depending on the structure of $\pi$. Either the digits of $\pi$ which make $i$ a separator of both types appear as one single part (like in $\pi=[624351]$, where $3$ is a separator of both types due to the sequence $2435$), or they appear in two different places (like in $\pi=[241536]$, where $24$ makes $3$ a horizontal separator and $5$ and $6$ make it a vertical one).

\begin{enumerate}
    \item The first case occurs when $i=2$ and $i$ is a part of the consecutive sub-sequence $13{\bf 2}4$ or its reverse. Similarly, when $i=n-1$ and $i$ is a part of the consecutive sub-sequence $n-3,{\bf n-1},n-2,n$ or its reverse.
    Also, for $2<i<n-1$, $i$ might be a part of one of the consecutive sub-sequences $i-2,{\bf i},i-1,i+1$ or $i-1,i+1,{\bf i},i+2$ or their reverses.\\ 
    \\ For each one of those sub-cases, there are exactly $(n-3)!$ ways to arrange the rest of the permutation.
    
    \item The permutation $\pi$ contains the sub-sequences of the form $a,i,a+1$ and $i-1,i+1$ or their reverses. 
    Again, we have to divide into two cases: $i\in \{2,n-1\}$ and $2< i<n-1$.
    
    If $i\in \{2,n-1\}$ then there are $n-4$ ways to choose $a$, while if $2< i<n-1$ then there are $n-5$ ways to choose $a$. Each choice of such two sub-sequences leaves $(n-3)!$ ways to arrange the rest of the permutation.
    
\end{enumerate}

Hence, we have
$$E[Y_i]=
\begin{cases}
0 & i=1,n  \\
\\
 \frac{2 \cdot (n-3)!+2\cdot 2\cdot (n-3)!(n-4)}{n!}  &i=2,n-1  \\\\
 \frac{ 4 \cdot (n-3)!+2\cdot 2\cdot (n-3)!(n-5)}{n!} & 3 \leq i \leq n-2
\end{cases} \\$$
We have now: $$E[Y]=\sum\limits_{i=1}^n{E[Y_i]}=2E[Y_2]+(n-4)E[Y_3].$$

A simple calculation now yields the following:
\begin{theorem}
The expectation of the number of separators of both types: vertical and horizontal, in a randomly chosen $n$-permutation is $$\frac{4(n-3)^2}{n(n-1)(n-2)}$$ The asymptotic value of the expectation is $0$. 
\end{theorem}
Now, let $Z$ be the random variable which counts the total number of separators (regardless of the type). By remark \ref{obs on Sep}.3, the number of vertical separators has the same distribution as the number of horizontal separators, so $E(Z)=2E(X)-E(Y)$.
So, we have the following:
\begin{theorem}
The expectation of the number of separators, in a randomly chosen $n$-permutation is $$\frac{4(n^3-6n^2+14n-13)}{n(n-1)(n-2)}.$$ The asymptotic value of the expectation is $4$. 
\end{theorem}

\end{document}